\documentclass{amsart}

\usepackage{amsmath, amsthm}
\usepackage{amsfonts}
\usepackage{amssymb}
\usepackage{amsthm}
\usepackage{amsmath}
\usepackage{latexsym}
\usepackage[all]{xy}
\usepackage{mathrsfs}
\usepackage{graphicx}

\newtheorem{theorem}{Theorem}
\newtheorem{corollary}[theorem]{Corollary}
\newtheorem{lemma}[theorem]{Lemma}

\title{A note on the inverse problem for the lattice points}

\author{\v Zeljka Ljuji\'c}
\address{Mathematics Ph.D. Program, The CUNY Graduate Center\\
         365 Fifth Avenue, Room 4208 New York, NY 10016-4309}
\email{zljujic@gc.cuny.edu}
\author{Camilo Sanabria}
\address{Mathematics Ph.D. Program, The CUNY Graduate Center\\
         365 Fifth Avenue, Room 4208 New York, NY 10016-4309}
\email{csanabria\_malagon@gc.cuny.edu}
\date{}
\thanks{The second author was partially supported by NSF grants CCF-0901175 and CCF-0964875}

\begin{document}

\begin{abstract}
Let $K\subseteq\mathbb{R}^2$ be a compact set such that $K+\mathbb{Z}^2=\mathbb{R}^2$. We prove, via Algebraic Topology, that the integer points of the difference set of $K$, $(K-K)\cap\mathbb{Z}^2$, is not contained on the coordinate axes, $\mathbb{Z}\times\{0\}\cup\{0\}\times\mathbb{Z}$. This result gives a negative answer to a question posed by P. Hegarty and M. Nathanson on relatively prime lattice points.
\end{abstract}

\maketitle

\section*{Introduction}

We consider the following context. Let $X$ be a metric space which is geodesic and proper, and let $\Gamma$ be a group. Let $\Gamma\times X\rightarrow X$ be a properly discontinuous action by isometries (from the left) such that the quotient $\Gamma\textrm{\textbackslash} X$ is compact. Such action is called \emph{geometric}. The proof of the fundamental observation of geometric group theory implies that in such a context, if $K\subseteq X$ is compact and a fundamental domain for the $\Gamma$-action then the set
\[
\{\gamma\in \Gamma|\ K\cap\gamma K\ne\emptyset\}
\]
is a finite set of generators of $\Gamma$ \cite{5}. This result was proved independently by V.A. Efremovi\v c \cite{6}, J. Milnor \cite{7} and A. S. \v Svarc \cite{8}.  In the case $X=\mathbb{R}^n$ and $\Gamma=\mathbb{Z}^n$ we obtain the following theorem:

\begin{theorem} If $K\subset\mathbb{R}^n$ is a compact set such that for every $x\in\mathbb{R}^n$ there exists $y\in K$ with $x\equiv y$ (mod $\mathbb{Z}^n$) in $\mathbb{R}^n$, then $A=(K-K)\cap\mathbb{Z}^n$ is a finite set of generators for $\mathbb{Z}^n$.
\end{theorem}

This result leads to the inverse problem, which was originally posed by M.B. Nathanson \cite{1}: If $A$ is a finite set of generators for a group $\Gamma$, such that $A$ is symmetric, i.e. $A^{-1}=A$, and contains the identity of $\Gamma$ does there exist a geometric action of $\Gamma$ on a metric space $X$ such that $A=\{\gamma\in \Gamma|\ K\cap\gamma K\ne\emptyset\}$ for some compact set $K$ which is a fundamental domain for such action? In the case $X=\mathbb{R}^n$ and $\Gamma=\mathbb{Z}^n$, this problem can be translated to a problem of number theory:

\emph{Which sets can be obtained as $(K-K)\cap\mathbb{Z}^n$ where $K$ is a compact set such that  for every $x\in\mathbb{R}^n$ there exists $y\in K$ with $x\equiv y$ (mod $\mathbb{Z}^n$)?}

This type of problems are one of the main topics of additive number theory. For further references see \cite{4}.

M.B. Nathanson in \cite{1} proved that a finite, symmetric set $A\subset \mathbb{Z}$, containing $0$, is a set of generators of $\mathbb{Z}$ if and only if there exist a compact set $K\subset\mathbb{R}$ such that $\mathbb{R}=K+\mathbb{Z}$ and $A=(K-K)\cap\mathbb{Z}$. This answers the inverse problem in the case $n=1$.

As an attempt to attack the case $n=2$, P. Hegarty raised the following question: 
Does there exist a compact set $K\subseteq\mathbb{R}^2$ such that for every $x\in\mathbb{R}^2$ there exists $y\in K$ with $x\equiv y$ (mod $\mathbb{Z}^2$) and $(K-K)\cap\mathbb{Z}^2\subset(\mathbb{Z}\times\{0\})\cup(\{0\}\times\mathbb{Z})$? In this paper we prove that the answer to this question is ``no''. This proves that the set $A=\{(-1,0), (0,-1), (0,0), (1,0), (0,1)\}$, although is a finite, symmetric set of generators of $\mathbb{Z}^2$ containing $0$, is not of the form $(K-K)\cap\mathbb{Z}^2$ for any compact fundamental domain $K$. One can easily see that this negative answer implies that in the case $n>1$ not every finite symmetric subset of generators $\mathbb{Z}^n$ containing $0$ is of the form $(K-K)\cap\mathbb{Z}^n$ for some compact $K\subseteq \mathbb{R}^n$ such that $\mathbb{R}^n=K+\mathbb{Z}^n$. This refines the inverse problem for $n>1$.

By using a different argument the same result was obtained by L.A. Borisov and R. Jin in \cite{2}.

\section*{The proof}

 We will start the proof by using the observation of R. Jin in \cite{2}, that says that instead of considering any compact set $K$, it is enough to consider a set $B=\bigcup_{i=0}^{n-1}\bigcup_{j=0}^{n-1}B_{i,j}+u_{i,j}$, where $B_{i,j}=[\frac{i}{n},\frac{i+1}{n}]\times[\frac{j}{n},\frac{j+1}{n}]$ and $u_{i,j}\in\mathbb{Z}^2$, for some $n$. For the sake of completeness we present the proof:
 
\begin{theorem} Let $K$ be a compact set of $\mathbb{R}^2$. For $J=(j_1,j_2)\in\mathbb{Z}^2$, let
\[B_{n,J}=[\frac{j_1}{n},\frac{j_1+1}{n}]\times[\frac{j_2}{n},\frac{j_2+1}{n}].
\]
\noindent There exists an integer $n_0$ such that for every integer $n\ge n_0$ there is a finite subset $\mathcal{J}$ of $\mathbb{Z}^2$ such that the set 
\[K_n=\bigcup_{J\in\mathcal{J}}B_{n,J}
\]
satisfies $K\subset K_n$ and
\[(K-K)\cap\mathbb{Z}^2=(K_n-K_n)\cap\mathbb{Z}^2.
\]
\end{theorem}

\begin{proof}
The set $K$ is compact, thus the set $K-K$ is compact. Indeed $K-K=f(K\times K)$, where $f$ is the continuous function $f:\mathbb{R}^2\rightarrow \mathbb{R}$, $(x,y)\mapsto x-y$, and the image of a compact set through a continuous map is a compact set. The compactness of $K-K$ implies that there is an $\varepsilon>0$ such that
\[
\min\{|x-y|: \ x\in K-K,\ y\in \mathbb{Z}^n\setminus(K-K)\}>\varepsilon,
\]
for the distance function $x\in K\mapsto d(x,\mathbb{Z}^2\setminus(K-K))\in\mathbb{R}$ is continuous and has a non-zero minimal value.

Let $n_0>4\sqrt{2}/\varepsilon$, and for $n\ge n_0$ set $\mathcal{J}=\{J\in\mathbb{Z}^2|\ B_{n,J}\cap K\ne\emptyset\}$. It is easy to see that $(K-K)\cap\mathbb{Z}^2\subseteq(K_n-K_n)\cap\mathbb{Z}^2$. To prove the other inclusion we will prove that $(\mathbb{Z}^2\setminus(K-K))\cap (K_n-K_n)=\emptyset$. Indeed, for any $z_1,z_2\in K_n$ there is a $x_1,x_2\in K$ such that $|x_1-z_1|<\sqrt{2}/n$ and $|x_2-z_2|<\sqrt{2}/n$. We have $|(x_1-x_2)-(z_1-z_2)|\le |x_1-z_1|+|x_2-z_2|<2\sqrt{2}/n$. So, for any $z\in K_n-K_n$ there is a $x\in K-K$ such that $|x-z|< 2\sqrt{2}/n$. Hence, for any $y\in \mathbb{Z}^2\setminus(K-K)$ and any $z\in K_n-K_n$,
\[
|y-z|\ge |y-x|-|x-z|>\varepsilon-2\sqrt{2}/n>\varepsilon/2>0.
\]
\end{proof}

We proceed to the proof of our main result.

\begin{theorem}  It does not exist a compact set $K$ s.t. $\mathbb{R}^2=K+\mathbb{Z}^2$ and $(K-K)\cap\mathbb{Z}^2\subseteq(\mathbb{Z}\times\{0\})\cup(\{0\}\times\mathbb{Z})$.
\end{theorem}

\begin{proof} Let us assume that such set $K$ exists. In view of the previous theorem there exist $n\in\mathbb{Z}_{>0}$ such that $(K-K)\cap\mathbb{Z}^2=(K_n-K_n)\cap\mathbb{Z}^2$. So, instead of working with $K$ we can work with $K_n$. Furthermore, it is enough to prove the statement for any subset of $K_n$, thus we may reduce $\mathcal{J}$ so that $|\mathcal{J}|=n^2$ and $K_n+\mathbb{Z}^2=\mathbb{R}^2$. We can write $K_n=\bigcup_{i=0}^{n-1}\bigcup_{j=0}^{n-1}B_{i,j}+u_{i,j}$, where $B_{i,j}=[\frac{i}{n},\frac{i+1}{n}]\times[\frac{j}{n},\frac{j+1}{n}]$ and $u_{i,j}\in\mathbb{Z}^2$. Translating $K_n$ by $-u_{0,0}$, we may assume $u_{0,0}=(0,0)$. We have $(K_n-K_n)\cap\mathbb{Z}^2\subseteq(\mathbb{Z}\times\{0\})\cup(\{0\}\times\mathbb{Z})$. 

Let us consider the unit square subdivided into $n^2$ squares $B_{i,j}$, where $0\leqslant i, j\leqslant n-1$. We label the vertices $(\frac{i}{n},\frac{j}{n})$, where $0\leqslant i, j\leqslant n$, with the value $v_{i,j}$ in the following way
\begin{equation*} 
v_{i,j} = \left\{ 
\begin{array}{rl} 
u_{i,j} & \text{for } 0\leqslant i, j\leqslant n-1\\ 
u_{0,j} +(-1,0)& \text{for } i=n,0\leqslant j \leqslant n-1\\ 
u_{i,0} +(0,-1)& \text{for } 0\leqslant i \leqslant n-1, j=n\\ 
(-1,-1)&\text{for } i=n,j=n 
\end{array} \right. 
\end{equation*} 

We direct the edges, the sides of the $B_{i,j}$'s, in upward and rightward direction and we label them with the value of the differences: value at the ending vertex minus value at the initial vertex. Note, that the unit square subdivided in this fashion, and with a prescribed orientation on the edges, can be seen as the torus $T$ with a given $\square$-complex structure. If we denote the labeling of the edges by $\psi$, then $\psi([v_{i,j}, v_{i+1,j}])=v_{i+1,j}-v_{i,j}$, for $0\leqslant i \leqslant n-1,0\leqslant j \leqslant n$ and $\psi([v_{i,j}, v_{i,j+1}])=v_{i,j+1}-v_{i,j}$, for $0\leqslant i \leqslant n,0\leqslant j \leqslant n-1$. Note that $\psi([v_{i,0}, v_{i+1,0}])=\psi([v_{i,n}, v_{i+1,n}])$, for $0\leqslant i \leqslant n-1$ and $\psi([v_{0,j}, v_{0,j+1}])=\psi([v_{n,j}, v_{n,j+1}])$, for $0\leqslant j \leqslant n-1$, so $\psi$ is a well-defined function from the edges of $T$ to the abelian group $\mathbb{Z}\times\mathbb{ Z}$. Moreover, $\psi([v_{i,j}, v_{i+1,j}])+\psi([v_{i+1,j}, v_{i+1,j+1}])-\psi([v_{i,j+1}, v_{i+1,j+1}])-\psi([v_{i,j}, v_{i,j+1}])=0$, for $0\leqslant i, j\leqslant n-1$, so we can see $\psi$ as one representative of an element of the cohomology group $H^1(T;\mathbb{Z}\times\mathbb{ Z})$. There is a natural map $h:H^1(T;\mathbb{Z}\times\mathbb{ Z})\rightarrow\textrm{Hom}(H_1(T),\mathbb{Z}\times\mathbb{ Z})$ that sends $\psi$ to $\overline{\psi_0}:H_1(T)\rightarrow\mathbb{Z}\times\mathbb{ Z}$ where $\overline{\psi_0}([[(0,0),(1,0)]])=(-1,0)$ and $\overline{\psi_0}([[(0,0),(0,1)]])=(0,-1)$. Here, we were using that $H_1(T)=\mathbb{Z}\times\mathbb{Z}$ with basis the homology classes $[[(0,0),(1,0)]]$ and $[[(0,0),(0,1)]]$. Hence, $\overline{\psi_0}$ is an isomorphism. This means that
\[
(*) \left. \begin{array}{l}\textrm{we can read the homotopy type of a closed curve from the sum} \\ \textrm{of the values associated by } \psi \textrm{ to the edges forming the curve.} \end{array} \right.
\]

All the values associated to the edges are lying in the set $K_n-K_n$. Therefore, we can color the edges in the following way:
red if  the value of the edge lays on the $x$-axis and it is different than $0$, blue if the value of the edge lays on the $y$-axis and it is different than $0$, and white if the value of the edge is $0$. Considering any four adjacent squares $B_{i,j}, B_{i+1,j}, B_{i,j+1}, B_{i+1,j+1}$, we can see that any of the squares $B_{i,j}$, for $0\leqslant i, j\leqslant n-2$, cannot have a red and a blue edge in the same time. In the case of the squares $B_{n-1,j}$, where $0\leqslant j\leqslant n-2$, and $B_{i,n-1}$, where $0\leqslant i\leqslant n-2$, the same conclusion arises from considering the squares  $B_{0,j}, B_{0,j+1}, B_{n-1,j}, B_{n-1,j+1}$ and $B_{i,0}, B_{i+1,0}, B_{i,n-1}, B_{i+1,n-1}$. The square $B_{n-1,n-1}$ cannot have red and blue edges at the same time, neither. Indeed, consider the squares $B_{0,0}, B_{0,n-1}, B_{n-1,0}, B_{n-1,n-1}$. We obtained that all the squares can have only red and white edges, blue and white, or all white edges. We will be calling them red, blue and white squares, respectively. Note that, by construction, no square can have only one red or only one blue edge. Also, the common edge between a red and a blue square is white.

We divide the unit square into red, blue and white components. By component, we mean a monochromatic collection of squares, maximal with respect to inclusion, such that when seen on the surface of the torus it is connected. 

Let $C$ be a component and $\sigma=\partial C$ the union of curves enclosing $C$. First, we prove that $\sigma$ is a union of closed curves. The proof is by induction on the number $n$ of squares in $C$. If $n=1$, the component is made of just one square, so $\sigma$ is the simple closed curve enclosing the square. Let $n\ge 2$, and suppose that any component having less than $n$ squares is enclosed by a union of closed curves. Let $C_0$ be a square in $C$ and let $C_1$ be the union of the squares in $C$ different than $C_0$. Then $C=C_0\cup C_1$. Moreover, $C_1$ can be seen as a disjoint union of components, each of them having less than $n$ squares, hence enclosed by a union of closed curves. Thus, $\partial C_1$ is the union of closed curves. There are five possibilities for the intersection $C_0\cap C_1$: it can be a vertex of $C_0$, a side of $C_0$ or a union of two, three or four sides of $C_0$. In each of the cases, we obtain that $\sigma$ is the union of closed curves. As any closed curve can be seen as the union of simple closed curves, we conclude that $\sigma$ is the union of simple closed curves. Note that $\sigma$ is also the topological boundary of $C$. We will refer to it as a boundary of $C$.

Given a curve, we call \emph{gain} the sum of the values associated by $\psi$ to the edges forming it. A gain of value $(\cdot,0)$ can only be obtained through red squares; a gain of value $(0,\cdot)$ can only be obtained through blue squares. Therefore, because the gain of the horizontal curve $[[(0,0),(1,0)]]$ is $(-1,0)$ and the gain of the vertical curve $[[(0,0),(0,1)]]$ is $(0,-1)$, the coloring must contain red, as well as blue component. The boundary of a component is a union of simple closed curves formed by white edges only. Whence, from $(*)$ above, the boundary of a component is a union of closed curves which are contractible on the torus. Indeed $\overline{\psi_0}$ on such white closed curves is zero. Here, by curve being contractible on the torus we mean that its homotopy class is $0$.

Let us consider any red component. If a component is contractible on the torus any horizontal line that crosses the component will have the horizontal gain equal to 0 inside the component. The horizontal gain $(-1,0)$ is obtained only through red squares, so there exists a red component that is non-contractible on the torus. On the other hand, it follows from the following lemma and its corollary, that if a component has a boundary that consists of simple closed curves which are contractible on the torus, then it is either contractible on the torus or it contains loops generating the fundamental group of the torus. Whence, there exist red component inside which we can obtain both gains, $(-1,0)$ and $(0,-1)$. A contradiction.

\end{proof}

\begin{lemma} Let $C$ be a component, and $i:C\rightarrow T$ be the inclusion map. Assume that  the boundary of $C$ is a single simple closed curve. Then
\[
i_*(\Pi_1(C))=\left\{
\begin{array}{l}
0 \\
\mathbb{Z}^2 
\end{array}\right.
\] 
\end{lemma}

\begin{proof} Let $f:I\rightarrow T$ be the boundary of C, where I is the unit interval $I=[0,1]$. We denote by $p:\mathbb{R}^2\rightarrow T=\mathbb{R}^2/\mathbb{Z}^2$ the universal cover of $T$. The interval $I$ is path-connected and locally path-connceted and $f_*(\pi_1(I))=0\subset0=p_*(\pi_1(\mathbb{R}^2))$, hence for each lift $\widetilde{f(0)}$ of $f(0)$, there is a unique path $\widetilde{f}:I\rightarrow\mathbb{R}^2$ lifting $f$ starting at $\widetilde{f(0)}$. This holds for any map  $g:J\rightarrow T$, where $J$ is an interval in $\mathbb{R}$. As $p$ is a covering map, it is a local homeomorphism of $\mathbb{R}^2$ with $T$, so $f$ being simple implies that the lifts are simple curves as well. Moreover, $[f]=0$, so there is a homotopy $F_t:I \rightarrow T$, $0\leq t\leq1$, such that $F_0=f$ and $F_1=f(0)$. By the homotopy lifting property, for each lift $\widetilde{f}$, there exists a unique homotopy $\widetilde{F_t}:I \rightarrow \mathbb{R}^2$ of $\widetilde{f}$ to $\widetilde{f(0)}$ that lifts $F_t$. Whence, $[\widetilde{f}]=0$ and $\widetilde{f}$ is a loop. Hence, every lift $\widetilde{f}$ of $f$ is a simple closed curve, so by Jordan curve theorem it separates $\mathbb{R}^2$ into two open, path-connected components, of which the image of $\widetilde{f}$ is the common boundary. 

We fix a lift $\widetilde{f(0)}$ and consider the lifting $\widetilde{f}:I\rightarrow\mathbb{R}^2$ of $f$ starting at $\widetilde{f(0)}$. Let us denote by $\widetilde{U}$ the interior region defined by $\widetilde{f}$ and by $\widetilde{D}=\overline{\widetilde{U}}=\widetilde{U}\cup\textrm{Im}(\widetilde{f})$ the closure. We will prove that $p|_{\widetilde{D}}$ is injective. Let us assume the contrary, so there exist $x, y\in \widetilde{D}$ such that $x\ne y$ and $p(x)=p(y)$. Hence, there exists $a\in\mathbb{Z}^2_{\ne(0,0)}$ such that $y=x+a$. On the other hand, $\widetilde{D}$ is a closure of a connected set, so it is connected and since it is locally path-connected, $\widetilde{D}$ is path-connected. Thus there exists a path $\widetilde{\gamma}:I\rightarrow \widetilde{D}$ with $\widetilde{\gamma}(0)=x$ and $\widetilde{\gamma}(1)=\widetilde{f(0)}$. We denote by $\widetilde{\delta}$ the closed curve 
\begin{equation*} 
 \widetilde{\delta}(t)= \left\{ 
\begin{array}{rl} 
\widetilde{\gamma}(t+1) & \text{for } -1\leqslant t\leqslant 0\\ 
\widetilde{f}(t)& \text{for } 0\leqslant t \leqslant 1\\ 
\widetilde{\gamma}(2-t)& \text{for } 1\leqslant t \leqslant 2
\end{array} \right. 
\end{equation*} 

\noindent We have $\widetilde{\delta}:[-1,2]\rightarrow\mathbb{R}^2$ and $\widetilde{\delta}(-1)=\widetilde{\delta}(2)=x$. Then $\delta=p\widetilde{\delta}:[-1,2]\rightarrow T$ is the loop with $\delta(-1)=\delta(2)=p(x)$. By assumption, $x$ and $y$ are two different lifts of $p(x)$, so we can consider the lift $\widetilde{\delta}$ of $\delta$ starting at $x$ and the lift $\widetilde{\delta}'$ of $\delta$ starting at $y$. We consider the closed curve $\widetilde{\delta}+a:I\rightarrow\mathbb{R}^2$. We have $p(\widetilde{\delta}+a)=\delta$ and $(\widetilde{\delta}+a)(-1)=y$, so by unique lifting property $\widetilde{\delta}'=\widetilde{\delta}+a$. Now, every lift of $\delta$ contains a lift of $f$, since $\delta(t)=f(t)$, for $0\leq t\leq 1$. Whence $\widetilde{f}(t)=\widetilde{\delta}(t)$, for $0\leq t\leq 1$ is the lift of $f$ starting at $\widetilde{f}(0)$ and $\widetilde{f}'(t)=\widetilde{\delta}'(t)$, for $0\leq t\leq 1$ is the lift of $f$ starting at $\widetilde{f}'(0)=\widetilde{f}(0)+a$. Furthermore, $\widetilde{f}'=\widetilde{f}+a$. Thus $\textrm{Im}(\widetilde{f})\cap \textrm{Im}(\widetilde{f'})=\emptyset$. For if $\widetilde{f}(t_1)=\widetilde{f'}(t_2)$, for some $0\leq t_1,t_2 \leq 1$, then $\widetilde{f}(t_1)=\widetilde{f}(t_2)+a$, hence $t_1\ne t_2$. Moreover, $f(t_1)=f(t_2)$, and since $f$ is a simple closed curve, we obtain $t_1=0, t_2=1$ or $t_1=1, t_2=0$, a contradiction, for $\widetilde{f}(0)=\widetilde{f}(1)=x\ne y=\widetilde{f}'(0)=\widetilde{f}'(1)$. We obtained that the images of the liftings $\widetilde{f}$ and $\widetilde{f}'=\widetilde{f}+a$ of $f$ are disjoint. We consider the intersection $\widetilde{D}\cap\widetilde{D'}$, where $\widetilde{D'}$ is the closure of the interior region defined by $\widetilde{f'}$. We have $\widetilde{D}'=\widetilde{D}+a$, so $\mu(\widetilde{D}')=\mu(\widetilde{D})$, where by $\mu$ we denote the Lebesgue measure. Since $\textrm{Im}(\widetilde{f})$ and $\textrm{Im}(\widetilde{f'})$ are disjoint, we have $\textrm{Im}(\widetilde{f})\subset\widetilde{U'}$ or $\textrm{Im}(\widetilde{f})\subset(\widetilde{D'})^C$. If $\textrm{Im}(\widetilde{f})\subset\widetilde{U'}$, then $\widetilde{D}\subset \widetilde{U'}\subset \widetilde{D'}$ and $\mu(\widetilde{D'}\setminus \widetilde{D})=\mu(\widetilde{D})-\mu(\widetilde{D'})=0$. Having that  $U'\setminus D\subset D'\setminus D$, we obtain $\mu(U'\setminus D)=0$. This is a contradiction, for $U'\setminus D$ is open in $\mathbb{R}^2$, whence $\mu(U'\setminus D)>0$. We obtain $\textrm{Im}(\widetilde{f})\subset(\widetilde{D'})^C$. Similarly, $\textrm{Im}(\widetilde{f'})\subset(\widetilde{D})^C$. Hence, $\widetilde{D}\cap\widetilde{D'}=\emptyset$. On the other hand, by assumption, $y\in\widetilde{D}$ and $y=x+a\in\widetilde{D}+a=\widetilde{D'}$. This is a contradiction, so $p|_{\widetilde{D}}$ is injective.

Next, we need to prove that if $D=p(\widetilde{D})$, then $C=D$ or $C=\overline{D^C}$. First, we prove that $\textrm{Int}(C)=C\setminus\textrm{Im}(f)$ and $\textrm{Int}(\overline{C^C})=\overline{C^C}\setminus\textrm{Im}(f)=C^C$ are connected sets in $T$. This is to say that $f$ divides $T$ into two connected components: $T\setminus\textrm{Im}(f)=\textrm{Int}(C)\cup \textrm{Int}(\overline{C^C})$. As this components are open in $\mathbb{R}^2$, they are locally path-connected, and thus path-connected. 

We consider $C\setminus\textrm{Im}(f)$. Proof is by induction on number $n$ of squares in $C$. If $n=1$, the component $C$ is a square and the square without its border is connected. Let $n\ge2$, and suppose that the statement is true if $C$ consists of less than $n$ squares. Let $C_0$ be a square in $C$ touching the boundary $\textrm{Im}(f)$ and let $C_1$ be the union of squares in $C$ different than $C_0$. Since $f$ is simple closed curves, the intersection $C_0\cap C_1$ can be one, two or three sides of $C_0$. In all three cases, the boundary $\partial C_1$ will be still closed simple curve. Thus, by induction hypothesis, $C_1\setminus\partial C_1$ is connected. Since the intersection $C_0\cap C_1$ is not contained in $\textrm{Im}(f)$, we obtain that $C\setminus\textrm{Im}(f)=(C_0\cup C_1)\setminus\textrm{Im}(f)$ is connected. A similar argument holds for $C^C$. 

Now, $p|_{\widetilde{D}}$ is injective, so $p(\widetilde{U})\cap p(\textrm{Im}(\widetilde{f}))=p(\widetilde{U})\cap\textrm{Im}(f)=\emptyset$, since $\widetilde{U}=\textrm{Int}(\widetilde{D})$. We have $p(\widetilde{U})\subset T\setminus\textrm{Im}(f)$. On the other hand, the interior $\widetilde{U}$ is connected. The covering map $p$ is continious, whence $p(\widetilde{U})$ is connected. We obtain that $p(\widetilde{U})\subset\textrm{Int}(C)$ or $p(\widetilde{U})\subset \textrm{Int}(\overline{C^C})$, or equivalently $D=p(\widetilde{D})\subset C$ or  $D=p(\widetilde{D})\subset \overline{C^C}$.

Let us prove that $D=C$ or $D=\overline{C^C}$, the latter being equivalent to $C=\overline{D^C}$. Let us assume that $D\subset C$. This means that $p(\widetilde{U})\subset\textrm{Int}(C)$. Fix $x\in p(\widetilde{U})$. There exists $\widetilde{x}\in \widetilde{U}$ such that $x=p(\widetilde{x})$. Let $y\in \textrm{Int}(C)$. Since $\textrm{Int}(C)$ is path-connected, there exist a path $g:I\rightarrow \textrm{Int}(C)\subset T$ such that $g(0)=x$ and $g(1)=y$. Let $\widetilde{g}:I\rightarrow\mathbb{R}^2$ be the unique path lifting $g$ and starting at $\widetilde{x}=\widetilde{g(0)}$. Then $\widetilde{y}=\widetilde{g(1)}$ is a lift of $y$, i.e. $p(\widetilde{y})=y$. Moreover, $\textrm{Im}(\widetilde{g})\cap \textrm{Im}(\widetilde{f})=\emptyset$. For, if $\widetilde{z}\in \textrm{Im}(\widetilde{g})\cap \textrm{Im}(\widetilde{f})$, then $p(\widetilde{z})\in \textrm{Im}(g)\cap \textrm{Im}(f)$, a contradiction, since $\textrm{Im}(g)\subset\textrm{Int}(C)$ and $\textrm{Int}(C)\cap\textrm{Im}(f)=\emptyset$. We obtain that $\textrm{Im}(\widetilde{g})\subset\widetilde{U}$, so $\widetilde{y}\in\widetilde{U}$ and $y\in p(\widetilde{U})$. Whence, $p(\widetilde{U})\subset\textrm{Int}(C)$ and $D=C$. Similarly, we conclude that if $p(\widetilde{U})\subset \textrm{Int}(\overline{C^C})$, then $C=\overline{D^C}$.

We are now in position of proving the statement. Let $h:I\rightarrow T$ be a closed curve in $D$. Then there is a unique lift $\widetilde{h}:I\rightarrow\mathbb{R}^2$ of $h$ such that $\widetilde{h}(0)\in\widetilde{D}$. Moreover, since $p|_{\widetilde{D}}$ is injective and $h(0)=h(1)$, we have $\widetilde{h}(0)=\widetilde{h}(1)$, so $\widetilde{h}$ is a closed curve in $\mathbb{R}^2$, hence $[\widetilde{h}]=0$. Having that $h=p(\widetilde{h})$, we obtain $[h]=[p\circ \widetilde{h}]=p_*(\widetilde{h})=0$. This means that if $j:D\rightarrow T$ denotes the inclusion map, then $j_*(\pi_1(D))=0$. But, we already proved that $C=D$ or $C=\overline{D^C}$. Thus, if $C=D$, then $i_*(\pi_1(C))=0$. On the other hand, if $C=\overline{D^C}$ , then as $j_*(\pi_1(D))=0$, Van Kampen's theorem implies  $i_*(\pi_1(C))=\mathbb{Z}^2$.

\end{proof}

\begin{corollary} Let $C$ be a component. Then
\[
i_*(\pi_1(C))=\left\{
\begin{array}{l}
0 \\
\mathbb{Z}^2 
\end{array}\right.
\] 

\end{corollary}

\begin{proof} First, we consider the case when the $\textrm{Int}(C)$ is connected. In this case, as $\textrm{Int}(C)$ is locally path-connected, whence $\textrm{Int}(C)$ is path-connected. We already noticed that the boundary of $C$ is a union of simple closed curves which are contractible on the torus. Let $f:I\rightarrow T$ be a simple closed curve that is a part of the boundary of $C$. Arguing as in the previous lemma, every lift $\widetilde{f}$ of $f$ is a simple closed curve and, by Jordan curve theorem, it separates $\mathbb{R}^2$ into two open, path-connected components, of which image of $\widetilde{f}$ is the common boundary. Let us fix a lift $\widetilde{f}$ of $f$ and let $\widetilde{U}$ be the interior region defined by $\widetilde{f}$. By the previous lemma, $p|_{\widetilde{D}}$ is injective, where $\widetilde{D}=\overline{\widetilde{U}}$. Two possibilities may occur: $p(\widetilde{U})\cap \textrm{Int}(C)\ne\emptyset$ or $p(\widetilde{U})\cap \textrm{Int}(C)=\emptyset$. Let us consider the case $p(\widetilde{U})\cap \textrm{Int}(C)\ne\emptyset$. As $\textrm{Int}(C)$ is path-connected, we obtain $\textrm{Int}(C)\subset p(\widetilde{U})$. Thus, we have $\textrm{Int}(C)\subset p(\widetilde{U})$ or $p(\widetilde{U})\cap \textrm{Int}(C)=\emptyset$, the latter being equivalent to $p(\widetilde{U})\subset C^C$. If $\textrm{Int}(C)\subset p(\widetilde{U})$, then $\textrm{Int}(C)\subset (\bigcup_{a\in\mathbb{Z}^2}p(\widetilde{U}+a))=p(\bigcup_{a\in\mathbb{Z}^2}(\widetilde{U}+a))$, where $\widetilde{U}+a$, when $a$ ranges through $\mathbb{Z}^2$, represents the interior regions of all liftings of $f$. Hence, $p^{-1}(\textrm{Int}(C))\subset\bigcup_{a\in\mathbb{Z}^2}(\widetilde{U}+a)$. On the other hand, if $p(\widetilde{U})\cap \textrm{Int}(C)=\emptyset$, then $p(\bigcup_{a\in\mathbb{Z}^2}(\widetilde{U}+a))\cap \textrm{Int}(C)=\emptyset$, so $p^{-1}(\textrm{Int}(C))\subset(\bigcup_{a\in\mathbb{Z}^2}(\widetilde{U}+a))^C=\bigcap_{a\in\mathbb{Z}^2}(\widetilde{U}+a)^C$. We conclude that $p^{-1}(C)\subset\bigcup_{a\in\mathbb{Z}^2}(\widetilde{D}+a)$ or $p^{-1}(C)\subset\bigcap_{a\in\mathbb{Z}^2}(\widetilde{U}+a)^C$. 

Now, each of the simple closed curves making the boundary of $C$ is lifted to simple closed curves through $p$ and each lift defines an interior and an exterior region. Let $D$ be the union of the closures of the interior regions and $E$ be the intersection of the closures of the exterior regions. By the previous argument, we obtain that $p^{-1}(C)\subset D$ or $p^{-1}(C)\subset E$, and since $p$ is surjective, we have $C\subset p(D)$ or $C\subset p(E)$. 

If $C\subset p(E)$, we actually have the equality $C=p(E)$. Indeed, if $p(E)\setminus C\ne\emptyset$, then exists a square $S$ in $p(E)$ not belonging to $C$ such that $S\cap C\ne\emptyset$. For, if that is not the case, $p(E)=(\bigcup_{S\in p(E)} S)\cup C$ would be a disconnection, which would contradict the fact that $p(E)$ is connected. But this would mean that there exists $x\in S$ such that $p^{-1}(x)\subset \bigcup_{a\in\mathbb{Z}^2}(\widetilde{U}+a)$, where $\widetilde{U}$ is an interior region of a lifting of one of the simple closed curves making the border of $C$. A contradiction, since there exists $y\in p^{-1}(x)$ such that $y\in E$ and $E\cap \bigcup_{a\in\mathbb{Z}^2}(\widetilde{U}+a)=\emptyset$.

On the other hand, by the previous lemma and Van Kampen's theorem, we have $j_*(\pi_1(p(D)))=0$, where $j:p(D)\rightarrow T$ is the inclusion map. Having that $D\cup E=\mathbb{R}^2$, we obtain that $p(D)\cup p(E)=T$ and, by Van Kampen's theorem, $k_*(\pi_1(p(E)))=\mathbb{Z}\times\mathbb{Z}$, where $k:p(E)\rightarrow T$ is the inclusion map. By the previous discussion, $C\subset p(D)$ or $C=p(E)$, which ends the proof in the case when $\textrm{Int}(C)$ is connected.

Finally, let us define the wedge sum as a union of two sets intersecting at only one point. Then any component $C$ can be seen as the wedge sum of components with connected interiors. The statement follows by Van Kampen's theorem.
  
\end{proof}


\begin{thebibliography}{99}
\bibitem{2} R. Jin, L.A. Borisov, Finding integral diagonal pairs in a two dimensional $\mathcal{N}$--set , arXiv:1007.1441.
\bibitem{6} V. A. Efremovi\v c, The proximity geometry of Riemannian manifolds, Uspekhi Mat. Nauk 8 (1953), 189.
\bibitem{5} P. de la Harpe, \emph{Topics in Geometric Group Theory}, Chicago Lectures in Mathematics, University of Chicago Press, Chicago, IL, 2000.
\bibitem{7} J. Milnor, A note on curvature and fundamental group, J. Differential Geometry 2 (1968), 1--7.
\bibitem{1} M. B. Nathanson, An inverse problem in number theory and geometric group theory, in: D. Chudnovsky and G. Chudnovsky, editors, \emph{Additive Number Theory}, Springer, New York, 2010 (available on  arXiv:0901.1458).
\bibitem{4} M. Nathanson, Additive Number Theory: Inverse Problems and Geometry of Sumsets, Graduate Text in Mathematics 165, Springer-Verlag, Berlin Heidelberg New York, 1996.
\bibitem{8} A. S. \v Svarc, A volume invariant of coverings, Dokl. Akad. Nauk SSSR (N.S.) 105 (1955), 32Ð34.

\end{thebibliography}
\end{document}